\documentclass[amssymb,12pt]{amsart}
\usepackage{amsmath}

\newtheorem{propo}{Proposition}[section]
\newtheorem{proposition}[propo]{Proposition}
\newtheorem{defi}[propo]{Definition}
\newtheorem{conje}[propo]{Conjecture}
\newtheorem{lemma}[propo]{Lemma}

\newtheorem{corollary}[propo]{Corollary}

\newtheorem{theorem}[propo]{Theorem}

\newcommand{\ZZ}{{\mathbb Z}}
\newcommand{\FF}{{\mathbb F}}

\newcommand{\QQ}{{\mathbb Q}}

\newcommand{\NN}{{\mathbb N}}

\begin{document}
\title{On $\textbf{On}_p$}
\author{Joseph DiMuro}
\address{Department of Mathematics and Computer Science\\
Biola University\\
La Mirada, CA, USA}
\email{joseph.dimuro@biola.edu}

\date{12/2/14}
\subjclass[2010]{12F05, 12F20}
\keywords{ordinals, fields, field extensions}

\begin{abstract}
Generalizing John Conway's construction of the Field $\textbf{On}_2$, we give the ``minimal'' definitions of addition and multiplication that turn the ordinals into a Field of characteristic $p$, for any prime $p$. We then analyze the structure of the resulting Field, which we will call $\textbf{On}_p$.
\end{abstract}

\maketitle

In Chapter 6 of \cite{Conway}, John Conway introduces the Field $\textbf{On}_2$: the Class of all ordinals with the appropriate addition and multiplication defined to obtain an algebraically closed Field of characteristic $2$. (Following Conway's convention, we will use capitalized terms like ``Group'', ``Ring'', and ``Field'' when the structure referred to is a proper Class.) The operations are referred to as
Nim-addition and Nim-multiplication, due to their connections with the game of Nim. Further descriptions of the structure of this Field are given by Lenstra in \cite{Lenstra1} and \cite{Lenstra2}.

On pg. 17 of \cite{Lenstra1}, Lenstra gives an addition operation which turns the Class of ordinals into an abelian group of exponent $3$. Lenstra then asks if there is an analogous definition of multiplication that produces a Field of characteristic $3$, and if other characteristics can similarly be handled. In \cite{Laubie}, Fran\c{c}ois Laubie gives an appropriate definition of addition for any prime characteristic $p$. (The definition is confined to the finite ordinals, but it works just as well for all ordinals.) In this paper, we will provide definitions of addition and multiplication that turn the ordinals into a Field of any prime characteristic $p$, which we will call $\textbf{On}_p$. We will also analyze the structure of these Fields $\textbf{On}_p$, obtaining results analogous to those in all the aforementioned references.

In what follows, we will normally use $+$ and $\times$ to denote addition and multiplication in $\textbf{On}_p$. If the standard ordinal operations are needed instead, then the given expression will be enclosed in brackets. For example, $[4(4)+3]=19$, but in $\textbf{On}_2$, $4(4)+3=6+3=5$. We will sometimes use exponentiation in a similar way: $[4^2+3]=19$, but in $\textbf{On}_2$, $4^2+3=4(4)+3=5$. Also, the notation $a\cdot b$ (when not in brackets) will only be used for repeated addition; in $\textbf{On}_2$, $2(4)=8$, but $2\cdot 4=4+4=0$.

\section{Addition and Multiplication in $\textbf{On}_p$}

In \cite{Conway}, the definitions of addition and multiplication in $\textbf{On}_2$ are given ``genetically'', as follows: for $\alpha,\beta\in \textbf{On}_2$, we have $$\alpha+\beta=\text{mex}\{\alpha'+\beta,\alpha+\beta'\}$$ and $$\alpha\beta=\text{mex}\{\alpha'\beta+\alpha\beta'-\alpha'\beta'\}.$$ Here, $\alpha'$ ranges over all ordinals less than $\alpha$, $\beta'$ ranges over all ordinals less than $\beta$, and ``mex'' represents the ``minimal excludent'' of the given set (i.e. the smallest ordinal not in that set). Thus, all sums and products are defined in terms of lexicographically earlier sums and
products.

In \cite{Laubie}, a similar genetic definition is given for addition in $\textbf{On}_p$. Unfortunately, finding a genetic definition for multiplication in $\textbf{On}_p$ is apparently much more difficult (reasons for this will be given later). Rather than working with genetic definitions, we will establish the structure of $\textbf{On}_p$ inductively, defining addition and multiplication on progressively larger fields.

Following the convention in \cite{Conway}, ordinals in $\textbf{On}_p$ will sometimes be treated as single elements of $\textbf{On}_p$, and will sometimes be treated as the set of all lesser ordinals. Thus, a single ordinal $\alpha\in \textbf{On}_p$ will be called a ``group'', or ``ring'', or ``field'',
whenever the set of ordinals $\beta<\alpha$ forms a group, or ring, or field.

For any given ordinal $\Delta$, the ``$\Delta$-th'' field in $\textbf{On}_p$ will be denoted by $\phi_\Delta$. In constructing these fields we must ensure that:

\begin{enumerate}
  \item
  For every ordinal $\Delta$, the ordinal $\phi_\Delta$ is a field of characteristic $p$.
  \item
  If $\Delta '<\Delta$, then $\phi_{\Delta '}<\phi_\Delta$.
  \item
  The operations on each field extend those of all previous fields. That is, if $\Delta '<\Delta$, and $\alpha,\beta\in\phi_{\Delta '}$, then $\alpha+\beta$ and $\alpha\beta$ are the same in both $\phi_{\Delta '}$ and $\phi_{\Delta}$.
\end{enumerate}

We will construct the fields by induction on $\Delta$, as follows:

\begin{itemize}
  \item
  If $\Delta=0$, then $\phi_\Delta=\phi_0=p$, and the operations on $p$ are just ordinary addition and multiplication modulo $p$. Thus, $\phi_0$ is isomorphic to $\FF_p$, the finite field of $p$ elements.
  \item
  If $\Delta$ is a successor ordinal, then we will construct $\phi_\Delta$ from $\phi_{[\Delta-1]}$, using the methods discussed below.
  \item
  If $\Delta$ is a limit ordinal, then let $\Delta[i]$ be a fundamental sequence of $\Delta$. Then $\phi_\Delta$ will be the limit ordinal whose fundamental sequence is the following: $\phi_\Delta[i]=\phi_{\Delta[i]}$. (In other words, $\phi_\Delta$ is the supremum of all previous fields.)
\end{itemize}

When $\Delta$ is a limit ordinal, the operations on $\phi_\Delta$ are already determined by induction: if $\alpha,\beta\in\phi_\Delta$, then we have $\alpha,\beta\in\phi_{\Delta[i]}$ for some ordinal $\Delta[i]$ in the fundamental sequence of $\Delta$. Then $\alpha+\beta$ and $\alpha\beta$ are defined to be the same in $\phi_\Delta$ as in $\phi_{\Delta[i]}$. Trivially, since every $\phi_{\Delta[i]}$ is a field of characteristic $p$, so is $\phi_\Delta$.

So the only remaining work is to define $\phi_\Delta$ when $\Delta$ is a successor ordinal. For simplicity in what follows, we will let $\widetilde{\phi}=\phi_\Delta$, and we will let $\phi$ be the previous field: $\phi=\phi_{[\Delta-1]}$.

\subsection{When $\phi$ is not algebraically closed}

Assume first that the field $\phi$ is not algebraically closed. Let $\phi[x]$ be the ring of polynomials with coefficients in $\phi$. Let $n$ be the smallest positive integer where not all polynomials in $\phi[x]$ of degree $n$ have roots in $\phi$. Let $h(x)\in\phi[x]$ be the ``lexicographically earliest'' polynomial such that $g(x)=x^n-h(x)$ has no root in $\phi$. In determining which polynomial is lexicographically earliest, we consider the coefficients of the largest power of $x$ first. (For example, $5x^3+2x^2+9x+17$ is
lexicographically earlier than $5x^3+3x^2+1$.) Note that $g(x)$ is then irreducible over $\phi$; if $g(x)$ were the product of two polynomials in $\phi[x]$ of degree less than $n$, then each of those polynomials would have a root in $\phi$, contradicting the fact that $g(x)$ has no root in $\phi$.

We then define the next field to be $\widetilde{\phi}=[\phi^n]$. The definitions of addition and multiplication on $\widetilde{\phi}$ will be chosen so that $\widetilde{\phi}$ is the extension of the field $\phi$ by a root of $g(x)$; the ordinal $\phi$ itself will serve as a root of $g(x)$.

Let $F$ be the factor ring $\phi[x]/\langle g(x)\rangle$. Because $g(x)$ is irreducible over $\phi$, $F$ is a field. Every element of $F$ is of the form $f(x)+\langle g(x)\rangle$, where $f(x)$ is a polynomial in $\phi[x]$ of degree less than $n$. That is, every element of $F$ has the form $\displaystyle\left(\sum_{i=0}^{n-1}x^i\alpha_i\right)+\langle g(x)\rangle$ for some ordinals $\alpha_i\in\phi$. Also, every element of $\widetilde{\phi}$ has a similar representation: if $\alpha\in\widetilde{\phi}$, then $\displaystyle\alpha=\left[\sum_{i=0}^{n-1}\phi^i\alpha_i\right]$ for some ordinals $\alpha_i\in\phi$. We thus have a one-to-one, onto map between $\widetilde{\phi}$ and $F$: we can define $\theta:\widetilde{\phi}\rightarrow F$ via $$\displaystyle\theta\left(\left[\sum_{i=0}^{n-1}\phi^i\alpha_i\right]\right)=\sum_{i=0}^{n-1}\left(x^i\alpha_i\right)+\langle g(x)\rangle.$$

We will use this map to directly define addition and multiplication in $\widetilde{\phi}$: if $\alpha,\beta\in\widetilde{\phi}$, then we let $\alpha+\beta=\theta^{-1}(\theta(\alpha)+\theta(\beta))$, and $\alpha\beta=\theta^{-1}(\theta(\alpha)\theta(\beta))$. We then have $\theta(\alpha+\beta)=\theta(\alpha)+\theta(\beta)$ and $\theta(\alpha\beta)=\theta(\alpha)\theta(\beta)$, so $\theta$ is an isomorphism.
Since $F$ is a field of characteristic $p$, so is $\widetilde{\phi}$.

Also, note that the given operations on $\widetilde{\phi}$ extend those on $\phi$. Given $\alpha,\beta\in\phi$, let $\gamma_1=\alpha+\beta$, $\gamma_2=\alpha\beta$ in $\phi$. Then in $\widetilde{\phi}$, we have $\alpha+\beta=\theta^{-1}(\theta(\alpha)+\theta(\beta))=\theta^{-1}((\alpha+\langle g(x)\rangle)+(\beta+\langle g(x)\rangle))=\theta^{-1}(\gamma_1+\langle g(x)\rangle)=\gamma_1$, and similarly $\alpha\beta=\gamma_2$. So as required, $\widetilde{\phi}$ is a field of characteristic $p$ extending the operations of $\phi$.

Finally, one more thing to note:

\begin{lemma}\label{powersofphi1}
Given a field $\phi\in On_p$, assume that the next field $\widetilde{\phi}$ is a degree $n$ extension of $\phi$. If $\displaystyle\alpha=\left[\sum_{i=0}^{n-1}\phi^i\alpha_i\right]\in\widetilde{\phi}$, then $\displaystyle\alpha=\sum_{i=0}^{n-1}\phi^i\alpha_i$.
\end{lemma}

\begin{proof}
We have $\displaystyle\theta\left(\left[\sum_{i=0}^{n-1}\phi^i\alpha_i\right]\right)=\sum_{i=0}^{n-1}\left(x^i\alpha_i\right)+\langle g(x)\rangle\newline =\sum_{i=0}^{n-1}\left(x+\langle g(x)\rangle\right)^i\alpha_i=\sum_{i=0}^{n-1}(\theta(\phi))^i\alpha_i=\theta\left(\sum_{i=0}^{n-1}\phi^i\alpha_i\right)$. Since $\theta$ is one-to-one, we must have $\displaystyle\left[\sum_{i=0}^{n-1}\phi^i\alpha_i\right]=\sum_{i=0}^{n-1}\phi^i\alpha_i$.
\end{proof}

\subsection{When $\phi$ is algebraically closed}

Now assume that $\phi$ is an algebraically closed field. The next field will be $\widetilde{\phi}=[\phi^\phi]$, and we will define addition and multiplication so that $\widetilde{\phi}$ is isomorphic to $\phi(x)$, the field of rational functions with coefficients in $\phi$.

Note: the smallest field in $\textbf{On}_p$ is $\phi_0=p$, and all fields $\phi_n$ for finite $n$ will be algebraic extensions of $\phi_0$, since no finite fields are algebraically closed. The next field must be $\phi_\omega=\omega$. Thereafter, all fields $\phi_\Delta$ will either be a [power] of the preceding field (when $\Delta$ is a successor ordinal) or the supremum of all previous fields (when $\Delta$ is
a limit ordinal). Thus, all infinite fields in $\textbf{On}_p$ will be [powers] of $\omega$, and hence, limit ordinals. (In what follows, we will need the fact that all algebraically closed fields are limit ordinals.)

Following the notation on pg. 62 in \cite{Conway}, any rational function $f(x)\in\phi(x)$ has a partial fraction expansion $$f(x)=\sum_i \frac{\beta_i}{(x-\alpha_i)^{n_i+1}}+\sum_jx^{m_j}\delta_j$$ where
$\alpha_i,\beta_i,\delta_j\in\phi$ and $n_i,m_j\in\omega$ for every $i$ and $j$. Also, every element in $\widetilde{\phi}=[\phi^{\phi}]$ has the form $$\left[\sum_i\phi^{\omega+\omega\alpha_i+n_i}\beta_i+\sum_j\phi^{m_j}\delta_j\right]$$ where $\alpha_i,\beta_i,\delta_j\in\phi$ and $n_i,m_j\in\omega$ for each $i$ and $j$. We thus have the following one-to-one, onto map: $\theta:\widetilde{\phi}\rightarrow\phi(x)$, where $$\theta\left(\left[\sum_i\phi^{\omega+\omega\alpha_i+n_i}\beta_i+\sum_j\phi^{m_j}\delta_j\right]\right)=\sum_i \frac{\beta_i}{(x-\alpha_i)^{n_i+1}}+\sum_jx^{m_j}\delta_j$$.

We define the operations on $\widetilde{\phi}$ the same way as in the case where $\phi$ is not algebraically closed: if $\alpha,\beta\in\widetilde{\phi}$, then we let $\alpha+\beta=\theta^{-1}(\theta(\alpha)+\theta(\beta))$, and $\alpha\beta=\theta^{-1}(\theta(\alpha)\theta(\beta))$. We then have $\theta(\alpha+\beta)=\theta(\alpha)+\theta(\beta)$ and $\theta(\alpha\beta)=\theta(\alpha)\theta(\beta)$, so $\theta$ is an isomorphism. Since $\phi(x)$ is a field of characteristic $p$, so is $\widetilde{\phi}$. And since $\theta(\alpha)=\alpha$ for every $\alpha\in\phi$, the operations in $\widetilde{\phi}$ extend those in $\phi$. So as required, $\widetilde{\phi}$ is a field of characteristic $p$ extending the operations of $\phi$.

As with the non-algebraically-closed case, we have one more thing to note:

\begin{lemma}\label{powersofphi2}
Given an algebraically closed field $\phi\in On_p$, let $\widetilde{\phi}$ be the next larger field. If $\displaystyle\alpha=\left[\sum\phi^\delta\alpha_\delta\right]\in\widetilde{\phi}$, then $\displaystyle\alpha=\sum\left[\phi^\delta\right]\alpha_\delta$.
\end{lemma}

This is essentially the same result as the previous lemma, except that the exponents on $\phi$ can now range over all ordinals $\delta<\phi$.

\begin{proof}
Given any finite collections of ordinals $\alpha_i,\beta_i,\delta_j\in\phi$ and $n_i,m_j\in\omega$, we have
$$\theta\left(\left[\sum_i\phi^{\omega+\omega\alpha_i+n_i}\beta_i+\sum_j\phi^{m_j}\delta_j\right]\right)=\sum_i \frac{\beta_i}{(x-\alpha_i)^{n_i+1}}+\sum_jx^{m_j}\delta_j$$
$$=\sum_i\theta\left(\left[\phi^{\omega+\omega\alpha_i+n_i}\right]\right)\beta_i+\sum_j\theta(\left[\phi^{m_j}\right])\delta_j=\theta\left(\sum_i\left[\phi^{\omega+\omega\alpha_i+n_i}\right]\beta_i+\sum_j\left[\phi^{m_j}\right]\delta_j\right)$$
Thus, since $\theta$ is one-to-one, we have:
$$\left[\sum_i\phi^{\omega+\omega\alpha_i+n_i}\beta_i+\sum_j\phi^{m_j}\delta_j\right]=\sum_i\left[\phi^{\omega+\omega\alpha_i+n_i}\right]\beta_i+\sum_j\left[\phi^{m_j}\right]\delta_j$$
\end{proof}

We now have defined addition and multiplication operations that turn the ordinals into $\textbf{On}_p$, a Field of characteristic $p$. But we still must show that these Fields, as defined, are the correct analogues of Conway's field $\textbf{On}_2$. We must show that these definitions of addition and multiplication are the ``minimal'' definitions which will turn the ordinals into a Field of characteristic $p$.

\section{The minimality of $\textbf{On}_p$}

As stated in the last section, the definitions of addition and multiplication in $\textbf{On}_2$ are given in \cite{Conway} as follows: $\alpha+\beta=\text{mex}\{\alpha'+\beta,\alpha+\beta'\}$, and $\alpha\beta=\text{mex}\{\alpha'\beta+\alpha\beta'-\alpha'\beta'\}$. These operations can be thought of as the ``minimal'' operations that turn the ordinals into a Field. That is, let's say we tried to write out an addition table for the ordinals, not filling in any entry until all lexicographically earlier entries are filled. (So we don't determine $\alpha+\beta$ until everything of the form $\alpha'+\beta$ or $\alpha+\beta'$ is determined.) For each entry in the table, we always choose the smallest ordinal which allows the resulting structure to be a Field. Then we do the same with multiplication, determining $\alpha\beta$ only after all products of the form $\alpha'\beta$, $\alpha\beta'$, and $\alpha'\beta'$ have been determined. The result would be Conway's definition of $\textbf{On}_2$, for the following reason:

\begin{lemma}\label{mexismin}
If $F$ is a Field consisting of the Class of all ordinals, then regardless of how addition and multiplication are defined, for all ordinals $\alpha,\beta$, we have $\alpha+\beta\ge\text{mex}\{\alpha'+\beta,\alpha+\beta'\}$, and $\alpha\beta\ge\text{mex}\{\alpha'\beta+\alpha\beta'-\alpha'\beta'\}$.
\end{lemma}

So, in $\textbf{On}_2$, each sum or product of ordinals yields the smallest ordinal possible, while still ensuring that $\textbf{On}_2$ is indeed a Field.

\begin{proof}
Assume that, for some ordinals $\alpha$ and $\beta$, we have $\alpha+\beta<\text{mex}\{\alpha'+\beta,\alpha+\beta'\}$. Then either $\alpha+\beta=\alpha'+\beta$ for some $\alpha'<\alpha$, or $\alpha+\beta=\alpha'+\beta$ for some $\beta'<\beta$. But then either $\alpha=\alpha'$ or $\beta=\beta'$, a contradiction.

Now assume that, for some ordinals $\alpha$ and $\beta$, we have $\alpha\beta<\text{mex}\{\alpha'\beta+\alpha\beta'-\alpha'\beta'\}$. Then for some $\alpha'<\alpha$ and $\beta'<\beta$, we have $\alpha\beta=\alpha'\beta+\alpha\beta'-\alpha'\beta'$. But then $(\alpha-\alpha')(\beta-\beta')=0$, a contradiction, since fields have no zero divisors.
\end{proof}

Now, in \cite{Laubie}, Laubie gives a genetic definition for $p$-adic addition in the set of finite ordinals (which can be extended to all ordinals). Under his definition, addition can be done by expressing ordinals in base $p$, then adding in base $p$ without carrying. (For example, in $\textbf{On}_3$, we have $22+19=(9+9+3+1)+(9+9+1)=9+3+1+1=14.$) We will now show that the definitions given in Section 1 produce the
same property of addition; thus, our non-genetic definition of addition is the same as Laubie's genetic definition.

\begin{theorem}
Given ordinals $\alpha,\beta\in \textbf{On}_p$, assume $\alpha=[\sum p^\delta a_\delta]$ and $\beta=[\sum p^\delta b_\delta]$ (where each $\delta$ is an ordinal, and every $a_\delta,b_\delta\in p$). Then $\alpha+\beta=[\sum p^\delta c_\delta]$, where each $c_\delta\equiv a_\delta+b_\delta$ (mod $p$).
\end{theorem}

\begin{proof}
There exists some field $\phi_\Delta$ containing both $\alpha$ and $\beta$. We will proceed by induction on $\Delta$. The statement is obviously true when $\Delta=0$, as then $\alpha,\beta\in p$, and addition in $p$ is just ordinary addition modulo $p$. If $\Delta$ is a limit ordinal, then for some ordinal $\Delta[i]$ in the fundamental sequence of $\Delta$, both $\alpha$ and $\beta$ are contained in $\phi_{\Delta[i]}$. So by induction, the statement is true in that case.

That leaves the case where $\Delta$ is a successor ordinal; let $\widetilde{\phi}=\phi_\Delta$, and let $\phi=\phi_{[\Delta-1]}$ be the previous field. Based on the work in Section 1, we have $\alpha=[\sum \phi^\delta \alpha_\delta]$ and $\beta=[\sum \phi^\delta \beta_\delta]$, where each $\alpha_\delta,\beta_\delta\in\phi$. (If $\widetilde{\phi}$ is a degree $n$ extension of $\phi$, then every $\delta<n$; if $\phi$ is
algebraically closed, then every $\delta\in\phi$.)  For each $\delta$, let $\alpha_\delta+\beta_\delta=\gamma_\delta$; by induction, each of these summations is just componentwise addition modulo $p$. Then, from Lemmas \ref{powersofphi1} and \ref{powersofphi2}, we have $\alpha+\beta=[\sum \phi^\delta \alpha_\delta]+[\sum \phi^\delta \beta_\delta]=\sum [\phi^\delta] \alpha_\delta+\sum [\phi^\delta] \beta_\delta=\sum [\phi^\delta]
\gamma_\delta=[\sum \phi^\delta \gamma_\delta]$. And since $\phi$ is a [power] of $p$, this is just componentwise addition modulo $p$.
\end{proof}

So indeed, our definition of addition is as it should be. And this leads to one further consequence: the elements of $\textbf{On}_p$ that are groups are exactly the ordinals of the form $[p^\alpha]$, for some ordinal $\alpha$.

We must now show why our definition of multiplication is the correct one.

\begin{defi}
Given ordinals $\alpha,\beta\in \textbf{On}_p$, we will say that the unordered pair $\{\alpha,\beta\}$ has the ``MEX property'' if $\alpha\beta=\text{mex}\{\alpha'\beta+\alpha\beta'-\alpha'\beta'\}$. We will call the set $\{\alpha'\beta+\alpha\beta'-\alpha'\beta'\}$ the ``MEX set'' of the unordered pair $\{\alpha,\beta\}$.
\end{defi}

Note: from Lemma \ref{mexismin}, we have $\alpha\beta\ge\text{mex}\{\alpha'\beta+\alpha\beta'-\alpha'\beta'\}$ for all $\alpha,\beta\in \textbf{On}_p$. Thus, we can say that $\{\alpha,\beta\}$ has the MEX property if $\alpha\beta\le\text{mex}\{\alpha'\beta+\alpha\beta'-\alpha'\beta'\}$. In other words, $\{\alpha,\beta\}$ has the MEX property if, for all $\gamma<\alpha\beta$, we have $\gamma=\alpha'\beta+\alpha\beta'-\alpha'\beta'$ for some choice of $\alpha'<\alpha$, $\beta'<\beta$. We will use this property in the lemmas that follow; to show that a pair
$\{\alpha,\beta\}$ has the MEX property, we will show that all ordinals less than $\alpha\beta$ can be written in the form $\alpha'\beta+\alpha\beta'-\alpha'\beta'$.

With this definition, we can say that $\{\alpha,\beta\}$ has the MEX property for all $\alpha,\beta\in \textbf{On}_2$. It turns out that the same does not apply for $\textbf{On}_p$ when $p\ge 3$ (as we will see later). However, we will prove that certain pairs $\{\alpha,\beta\}$ do have the MEX property, and that these particular products determine the products of any two groups in $\textbf{On}_p$. (Thus, they determine the entire multiplication table of $\textbf{On}_p$, since the products of arbitrary elements can be determined from the products of groups via the distributive law.) That should be sufficient evidence that our definition of multiplication is the ``correct'' one.

\begin{lemma}
Let $\phi\in \textbf{On}_p$ be a field, and let $\widetilde{\phi}$ be the next larger field in $\textbf{On}_p$. Let $\alpha,\beta\in \textbf{On}_p$ be ordinals such that $\alpha>0$, $[\phi^\alpha]\in\widetilde{\phi}$ and $\beta\in\phi$. Then the pair $\{[\phi^\alpha],\beta\}$ has the MEX property.
\end{lemma}

\begin{proof}
Based on Lemmas \ref{powersofphi1} and \ref{powersofphi2}, we have $[\phi^\alpha]\beta=[\phi^\alpha\beta]$, regardless of whether $\phi$ is algebraically closed or not. If $\gamma<[\phi^\alpha\beta]$, then $\gamma=[\phi^\alpha\beta_1+\beta_2]=[\phi^\alpha]\beta_1+\beta_2$ for some $\beta_1<\beta$ and some $\beta_2\in[\phi^\alpha]$. To show that $\gamma$ is in the MEX set of $\{[\phi^\alpha],\beta\}$, we must show that $\gamma=\alpha_1\beta+[\phi^\alpha]\beta'-\alpha_1\beta'$ for some ordinals $\alpha_1<\phi^\alpha$, $\beta'<\beta$.

Let $\beta'=\beta_1$. Since $\beta-\beta_1$ is a nonzero element of $\phi$, and since $\phi$ is a field, there exists some $\phi'\in\phi$ where $\phi'(\beta-\beta_1)=1$. Let $\alpha_1=\phi'\beta_2\in[\phi^\alpha]$. We then have
$\alpha_1\beta+[\phi^\alpha]\beta'-\alpha_1\beta'=[\phi^\alpha]\beta'+\alpha_1(\beta-\beta')=[\phi^\alpha]\beta_1+\phi'\beta_2(\beta-\beta_1)=[\phi^\alpha]\beta_1+\beta_2=\gamma$. So we have rewritten $\gamma$ as required; $\{[\phi^\alpha],\beta\}$ has the MEX property.
\end{proof}

\begin{lemma}
Let $\phi\in \textbf{On}_p$ be a field that is not algebraically closed, and assume that $\widetilde{\phi}\in \textbf{On}_p$, the next larger field, is an extension of $\phi$ of degree $n$. Let $i,j$ be nonnegative integers such that $i+j\le n$. Then the pair $\{[\phi^i],[\phi^j]\}$ has the MEX property.
\end{lemma}

\begin{proof}
Let $h(x)\in\phi[x]$ be the lexicographically earliest polynomial such that $g(x)=x^n-h(x)$ has no root in $\phi$.

If $i+j<n$, then $[\phi^i][\phi^j]=[\phi^{i+j}]$. If $\gamma<[\phi^{i+j}]$, then for some polynomial $m(x)\in\phi[x]$ of degree less than $i+j$, we have $\gamma=[m(\phi)]$. Let $f(x)=x^{i+j}-m(x)$; then $f(x)$ is a monic polynomial of degree less than $n$. So $f$ can be factored into linear factors over $\phi$. Assume $f(x)=f_1(x)f_2(x)$, where $f_1(x),f_2(x)\in\phi[x]$, $f_1$ is monic of degree $i$, and $f_2$ is monic of degree $j$.

Let $m_1(x)=x^i-f_1(x)$ (which has degree less than $i$), and let $m_2(x)=x^j-f_2(x)$ (which has degree less than $j$). Then $m(x)=x^{i+j}-f(x)=x^{i+j}-(x^i-m_1(x))(x^j-m_2(x))=x^im_2(x)+x^jm_1(x)-m_1(x)m_2(x)$. So we have $\gamma=[\phi^i][m_2(\phi)]+[m_1(\phi)][\phi^j]-[m_1(\phi)][m_2(\phi)]$, and we've rewritten $\gamma$ in the desired form. So $\{[\phi^i],[\phi^j]\}$ has the MEX property if $i+j<n$.

Now assume $i+j=n$. Then $[\phi^i][\phi^j]=[h(\phi)]$. If $\gamma<[h(\phi)]$, then $\gamma=[m(\phi)]$ for some polynomial $m(x)\in\phi[x]$ that is lexicographically earlier than $h(x)$. Let $f(x)=x^{i+j}-m(x)$; then $f(x)$ can be factored into linear factors over $\phi$. We then proceed as before; let $f(x)=f_1(x)f_2(x)$, where $f_1$ is monic of degree $i$, and $f_2$ is monic of degree $j$. Then if
$m_1(x)=x^i-f_1(x)$ (which has degree less than $i$) and $m_2(x)=x^j-f_2(x)$ (which has degree less than $j$), then we have $\gamma=[\phi^i][m_2(\phi)]+[m_1(\phi)][\phi^j]-[m_1(\phi)][m_2(\phi)]$. We've rewritten $\gamma$ in the desired form, so $\{[\phi^i],[\phi^j]\}$ has the MEX property if $i+j=n$.
\end{proof}

This essentially establishes that we have the correct definition of multiplication in $\widetilde{\phi}$, when the preceding field $\phi$ is not algebraically closed, and multiplication in $\phi$ has already been determined. If $\widetilde{\phi}=[\phi^n]$, then all pairs $\{[\phi^i],[\phi^j]\}$ (where $i+j\le n$) satisfy the MEX property. And by induction, the products of such pairs determine all products
$[\phi^i][\phi^j]$ when $i+j>n$; if $\displaystyle[\phi^i][\phi^j]=\sum _{k=0}^{n-1}\phi^k\gamma_k$ is known, and if $j<n$, then we are forced to have $\displaystyle[\phi^i][\phi^{j+1}]=[\phi^i][\phi^j]\phi=\left(\sum _{k=0}^{n-1}\phi^k\gamma_k\right)\phi=\sum _{k=0}^{n-1}\phi^{k+1}\gamma_k$. Finally, if $\alpha,\beta\in\widetilde{\phi}$ are groups, then $\alpha=[\phi^ia]$ and $\beta=[\phi^jb]$, where $0\le i,j<n$, and $a,b\in\phi$ are
groups. Then $\alpha\beta$ is determined by all the aforementioned rules: $\displaystyle\alpha\beta=[\phi^ia][\phi^jb]=[\phi^i][\phi^j](ab)=\left(\sum _{k=0}^{n-1}\phi^k\gamma_k\right)(ab)=\left[\sum _{k=0}^{n-1}\phi^kc_k\right]$, where $c_k=\gamma_kab$ (a product in $\phi$). So the products of all groups in $\widetilde{\phi}$ are determined by the above rules, and thus, so is the entire multiplication table of $\widetilde{\phi}$ (by the
distributive law).

We still must consider the case where $\phi$ is algebraically closed.

\begin{lemma}
If $\phi\in \textbf{On}_p$ is an algebraically closed field, then for all $i,j\in\omega$, the pair $\{[\phi^i],[\phi^j]\}$ has the MEX property.
\end{lemma}

The proof is essentially the same as for the $i+j<n$ case in the previous lemma. So, the products $[\phi^i][\phi^j]=[\phi^{i+j}]$ all satisfy the MEX property, and they establish the fact that $[\phi^i]=\phi^i$ for all $i\in\omega$.

\begin{lemma}
If $\phi\in \textbf{On}_p$ is an algebraically closed field, then for all $\alpha\in\phi$, the pair $\{[\phi^{\omega+\omega\alpha}],\phi\}$ has the MEX property.
\end{lemma}

\begin{proof}
We have $[\phi^{\omega+\omega\alpha}]\phi=(\frac{1}{\phi-\alpha})\phi=\frac{\alpha}{\phi-\alpha}+1=[\phi^{\omega+\omega\alpha}\alpha+1]$. If $\gamma<[\phi^{\omega+\omega\alpha}\alpha+1]$, then either $\gamma=[\phi^{\omega+\omega\alpha}\alpha]=\frac{\alpha}{\phi-\alpha}$, or $\gamma=\frac{\alpha'}{\phi-\alpha}+f(\phi)$, where $\alpha'<\alpha$, and $f(x)\in\phi(x)$ is a rational function whose poles are all less than $\alpha$.

The typical element of the MEX set of $\{[\phi^{\omega+\omega\alpha}],\phi\}$ has the form
$g(\phi)\phi+\frac{1}{\phi-\alpha}\phi'-g(\phi)\phi'$, where $\phi'<\phi$, and $g(x)\in\phi(x)$ is a rational function whose poles are all less than $\alpha$. We may set $\phi'=\alpha'$, obtaining
$g(\phi)\phi+\frac{1}{\phi-\alpha}\alpha'-g(\phi)\alpha'=\frac{\alpha'}{\phi-\alpha}+g(\phi)(\phi-\alpha')$. And this will equal $\frac{\alpha'}{\phi-\alpha}+f(\phi)$ when $g(x)=\frac{f(x)}{x-\alpha'}$ (which is indeed a rational function whose poles are all less than $\alpha$). That covers all possible values of $\gamma$ except $\gamma=\frac{\alpha}{\phi-\alpha}$, which we may obtain by setting $\phi'=\alpha$ and $g(x)=0$. So every value of $\gamma$ is in the MEX set; $\{[\phi^{\omega+\omega\alpha}],\phi\}$ has the MEX property.
\end{proof}

So, the products $[\phi^{\omega+\omega\alpha}]\phi=[\phi^{\omega+\omega\alpha}\alpha+1]$ all satisfy the MEX property, and they establish the fact that, for all $\alpha\in\phi$, $[\phi^{\omega+\omega\alpha}]=\frac{1}{(\phi-\alpha)}$.

\begin{lemma}
If $\phi\in \textbf{On}_p$ is an algebraically closed field, then for all $\alpha\in\phi$ and all positive integers $n$, the pair $\{[\phi^{\omega+\omega\alpha+n}],\phi\}$ has the MEX property.
\end{lemma}

\begin{proof}
We have $[\phi^{\omega+\omega\alpha+n}]\phi=\frac{1}{(\phi-\alpha)^{n+1}}\phi=\frac{1}{(\phi-\alpha)^n}\frac{\phi}{\phi-\alpha}
=\frac{1}{(\phi-\alpha)^n}(\frac{\alpha}{\phi-\alpha}+1)=\frac{\alpha}{(\phi-\alpha)^{n+1}}+\frac{1}{(\phi-\alpha)^n}
=[\phi^{\omega+\omega\alpha+n}\alpha+\phi^{\omega+\omega\alpha+n-1}]$. If $\gamma$ is a smaller ordinal, then $\gamma$ must take one of the following two forms:

1. $\gamma=\frac{\alpha}{(\phi-\alpha)^{n+1}}+f(\phi)$, where all poles of $f(x)\in\phi(x)$ are at most $\alpha$, and if $\alpha$ is a pole of $f(x)$, then it has degree at most $n-1$.

2. $\gamma=\frac{\alpha'}{(\phi-\alpha)^{n+1}}+f(\phi)$, where $\alpha'<\alpha$, all poles of $f(x)\in\phi(x)$ are at most $\alpha$, and if $\alpha$ is a pole of $f(x)$, then it has degree at most $n$.

Meanwhile, the typical element of the MEX set of $\{[\phi^{\omega+\omega\alpha+n}],\phi\}$ has the form
$g(\phi)\phi+\frac{1}{(\phi-\alpha)^{n+1}}\phi'-g(\phi)\phi'$, where $\phi'<\phi$, all poles of $g(x)\in\phi(x)$ are at most $\alpha$, and if $\alpha$ is a pole of $g(x)$, then it has degree at most $n$. This typical element of the MEX set will equal $\frac{\alpha}{(\phi-\alpha)^{n+1}}+f(\phi)$ when $\phi'=\alpha$ and $g(x)=\frac{f(x)}{x-\alpha}$. And this typical element of the MEX set will equal $\frac{\alpha'}{(\phi-\alpha)^{n+1}}+f(\phi)$ when $\phi'=\alpha'$ and $g(x)=\frac{f(x)}{x-\alpha'}$. So all values of $\gamma$ are in the MEX set; $\{[\phi^{\omega+\omega\alpha+n}],\phi\}$ has the MEX property.
\end{proof}

So, the products $[\phi^{\omega+\omega\alpha+n}]\phi=[\phi^{\omega+\omega\alpha+n}\alpha+\phi^{\omega+\omega\alpha+n-1}]$ all satisfy the MEX property. And they establish the fact that, for all $\alpha\in\phi$, $[\phi^{\omega+\omega\alpha+n}]=\frac{1}{(\phi-\alpha)^{n+1}}$.

These facts are enough to determine the products of all groups in $[\phi^\phi]$ (and hence, to determine the product of all elements of $[\phi^\phi]$), assuming that multiplication in $\phi$ has already been determined. Any group in $[\phi^\phi]$ has the form $[\phi^\alpha\beta]$, where $\alpha,\beta\in\phi$, and $\beta$ is a group. We have $[\phi^\alpha\beta]=[\phi^\alpha]\beta$, and based on the above lemmas, $[\phi^\alpha]=f(\phi)$ for some $f(x)\in\phi(x)$. Thus, $[\phi^\alpha]\beta=f(\phi)\beta$, and the product of two such rational functions of $\phi$ is determined. So multiplication in $\textbf{On}_p$ is completely determined by this collection of products that satisfy the MEX property.

However, there are pairs of elements in $\textbf{On}_p$ (for $p\ge 3$) that do not satisfy the MEX property. The simplest exception: in $\textbf{On}_3$, we have $3(3)=2$ (as we will see later), so $4(4)=(3+1)(3+1)=2+3+3+1=6$. But the minimal excludent of the set $\{\alpha'\beta+\alpha\beta'-\alpha'\beta'\}$, where $\alpha'$ and $\beta'$ range over all ordinals less than $4$, can be shown to be $2$. So we certainly cannot use $\alpha\beta=\text{mex}\{\alpha'\beta+\alpha\beta'-\alpha'\beta'\}$ as a genetic definition of multiplication in $\textbf{On}_p$.

It might be suspected that, if not all pairs of elements in $\textbf{On}_p$ satisfy the MEX property, then perhaps all pairs of \textbf{groups} in $\textbf{On}_p$ satisfy the MEX property. We would then be able to extend from products of groups to the full multiplication table, via the distributive property. But this also turns out to be false. In fact, there are pairs of \textbf{rings} that don't satisfy the MEX
property. (However, from the previous lemmas, all pairs of fields do satisfy the MEX property.)

\begin{theorem}
For any prime $p\ge 3$, if $\phi\in \textbf{On}_p$ is an algebraically closed field, then the pair $\{[\phi^{\omega\cdot 2}],[\phi^\omega]\}$ does not satisfy the MEX property.
\end{theorem}

\begin{proof}
We have $[\phi^{\omega\cdot 2}][\phi^\omega]=\frac{1}{\phi-1}\frac{1}{\phi}=\frac{1}{\phi-1}-\frac{1}{\phi}=\frac{1}{\phi-1}+\frac{[p-1]}{\phi}=[\phi^{\omega\cdot 2}+\phi^\omega(p-1)]$. We will show that $\gamma=[\phi^{\omega\cdot 2}+\phi^\omega]=\frac{1}{\phi-1}+\frac{1}{\phi}$ is not in the MEX set of $\{[\phi^{\omega\cdot 2}],[\phi^\omega]\}$.

An arbitrary element of the MEX set of $\{[\phi^{\omega\cdot 2}],[\phi^\omega]\}$ has the form
$g_1(\phi)\frac{1}{\phi}+\frac{1}{\phi-1}g_0(\phi)-g_1(\phi)g_0(\phi)$, where $g_0(x),g_1(x)\in\phi(x)$, $g_0(x)$ is a polynomial, and $g_1(x)$ has no poles except for possibly 0. We must choose appropriate rational functions $g_0(x)$ and $g_1(x)$ so that $f(x)=g_1(x)(\frac{1}{x}-g_0(x))+\frac{1}{x-1}g_0(x)$ equals $\frac{1}{x-1}+\frac{1}{x}$.

If $g_1(x)$ has a pole at 0, then $f(x)$ would have a pole at 0 of degree at least two. So $g_1(x)$ must be a polynomial. Furthermore, if $g_1(x)=a+xh_1(x)$, and $g_0(x)=b+(x-1)h_0(x)$, then we would have $f(x)=\frac{b}{x-1}+\frac{a}{x}+h(x)$ for some polynomial $h(x)$. So we must have $a=b=1$. Plugging those things in and simplifying, we find that $f(x)=\frac{1}{x-1}+\frac{1}{x}+h(x)$, where $$h(x)=(-x^2+x)h_1(x)h_0(x)+(-x+1)h_1(x)+(-x+2)h_0(x)-1.$$ But it is then impossible to have $h(x)=0$; if either $h_1(x)$ or $h_0(x)$ is a
nonzero polynomial, then $h(x)$ would be a polynomial of degree at least one. And if $h_1(x)=h_0(x)=0$, then we would have $h(x)=-1$. So regardless of the choice of $g_1(x)$ and $g_0(x)$, $f(x)$ cannot equal $\frac{1}{x-1}+\frac{1}{x}$.

Thus, $\gamma$ is not in the MEX set of $\{[\phi^{\omega\cdot 2}],[\phi^\omega]\}$; so the pair $\{[\phi^{\omega\cdot 2}],[\phi^\omega]\}$ does not satisfy the MEX property.
\end{proof}

All this raises the question: is there a purely genetic definition of multiplication in $\textbf{On}_p$? Given that the property $\alpha\beta=\text{mex}\{\alpha'\beta+\alpha\beta'-\alpha'\beta'\}$ does not even hold for rings, finding a genetic definition of multiplication would seem to be very difficult. So we shall have to rely on the inductive definitions.

\section{The algebraic closure of $p$ in $\textbf{On}_p$}

We will now focus on the structure of $\textbf{On}_p$ below the first transcendental. As we've seen, all ordinals that are fields, but not algebraically closed, will define algebraic extensions of themselves. All ordinals that are algebraically closed fields define transcendental extensions of themselves; if $\phi$ is an algebraically closed field, then since $\phi$ is not an element of itself, $\phi$ must be transcendental over itself! So we will refer to such elements $\phi$ as ``transcendentals'' in $\textbf{On}_p$.

In \cite{Conway}, Conway describes the general structure on $\textbf{On}_2$ below the first transcendental, $[\omega^{\omega^\omega}]$. For example, the first fields in $\textbf{On}_2$ are of the form $[2^{2^n}]$, and each is a quadratic extension of the previous one: $2^2=3$, $4^2=6$, $16^2=24$, $256^2=384$, and so on. (Each field, when squared, produces the [sesquimultiple] of that field.) The next fields are of the
form $[\omega^{3^n}]$, and each is a cubic extension of the previous one: $\omega^3=2$, $[\omega^3]^3=\omega$, $[\omega^9]^3=[\omega^3]$, etc. Then we have the quintic extensions, and so on. In this section, we will see that for any prime $p$, the field $\textbf{On}_p$ has a similar structure below the first transcendental.

Borrowing the notation in \cite{Lenstra2}, if $r=[u^n]$ is a prime [power] ($u$ prime, $n$ a positive integer), and if $k$ is the number of primes less than $u$ (so $k=0$ if $u=2$, $k=1$ if $u=3$, etc.), then we write $\chi_r=[p^{(\omega^k u^{n-1})}]$. Note that we then have $\chi_r=[\omega^{(\omega^{k-1} u^{n-1})}]$ if $k\ge 1$, so there will be no ambiguity in writing $\chi_r$ without reference to $p$ as long
as $u\ge 3$. (If $u=2$, and if there is a chance of ambiguity, we will write $\chi_{r,p}$ for $[p^{2^{n-1}}]$.)

Note that we have $\chi_{[u_1^{n_1}]}<\chi_{[u_2^{n_2}]}$ if and only if either $u_1<u_2$, or $u_1=u_2$ and $n_1<n_2$. So, we have $\chi_2\subseteq\chi_4\subseteq\chi_8\subseteq\cdots\subseteq\chi_3\subseteq\chi_9\subseteq\cdots\subseteq\chi_5\subseteq\chi_{25}\subseteq\cdots$. Also, the supremum of all the elements $\chi_r$ is $[p^{\omega^\omega}]=[\omega^{\omega^\omega}]$.

Our main goal of this section is to prove the following theorem.

\begin{theorem}\label{closureofp}
The following hold for all primes $p$:
\begin{enumerate}
  \item
  The first transcendental in $\textbf{On}_p$ is $[\omega^{\omega^\omega}]$.
  \item
  The ordinals in $[\omega^{\omega^\omega}]\in \textbf{On}_p$ that are fields are exactly the $\chi_r$ for prime [powers] $r$.
  \item
  For each prime $u$ and each integer $n\ge 2$, $\chi_{[u^n]}$ is a $u$th degree extension of $\chi_{[u^{n-1}]}$. Also, $\chi_{[u^n]}$ is closed under all extensions of degree $u'$, for any prime $u'<u$.
  \item
  If $p\ne u$, then we have $(\chi_u)^u=\alpha_u$, where $\alpha_u$ is the smallest ordinal in $\chi_u$ with no $u$th root in $\chi_u$. 
  \item
  We have $(\chi_{[u^{n+1}]})^u=\chi_{[u^n]}$ whenever $p\ne u$ and $n\ge 1$, except in the case where $[u^{n+1}]=4$ and $p\equiv 3$ mod $4$.
  \item
  We have $(\chi_4)^2=\chi_2+1$ when $p\equiv 3$ mod $4$.
  \item
  Finally, we have the case where $p=u$: for $n\ge 2$, we have $\displaystyle(\chi_{[p^n]})^p=\chi_{[p^n]}+\prod_{k=1}^{n-1}(\chi_{[p^k]})^{[p-1]}=\chi_{[p^n]}+\left[(\chi_p)^{p^n-1}\right]$. If $n=1$, then we have $(\chi_p)^p=\chi_p+1$.
\end{enumerate}
\end{theorem}

For example, consider the structure of $\textbf{On}_3$ below the first transcendental. The first extensions are all by square roots: we have $3^2=2$, $9^2=4$ (not $3$, since $3\equiv 3$ mod $4$), $81^2=9$, $6561^2=81$, and so on. Then we have the cubic extensions; we have $\omega^3=\omega+1$, $[\omega^3]^3=[\omega^3+\omega^2]$, $[\omega^9]^3=[\omega^9+\omega^8]$, $[\omega^{27}]^3=[\omega^{27}+\omega^{26}]$, and so on. Then come the quintic extensions: $[\omega^\omega]^5=10$, $[\omega^{\omega\cdot 5}]^5=[\omega^\omega]$, $[\omega^{\omega\cdot 25}]^5=[\omega^{\omega\cdot 5}]$, etc. And this pattern continues throughout all extensions. (In Section 4, we will look at methods for determining the values of $\alpha_u$ for $u\ne p$: in $\textbf{On}_3$, we have $\alpha_2=2$, $\alpha_5=10$, etc.)

We will prove Theorem \ref{closureofp} by a sequence of lemmas.

\begin{lemma}
Every element of $\textbf{On}_p$ below the first transcendental is contained in a finite field within $\textbf{On}_p$.
\end{lemma}

True by induction: given an element $\alpha$ below the first transcendental, let $\chi$ be the smallest ordinal where $\chi$ is a field and $\chi>\alpha$. If $\chi$ is finite, then $\chi$ is the finite field we want; so assume $\chi$ is infinite. From the work in Section 1, $\chi$ is an algebraic extension of a smaller field $\chi'$, so $\alpha$ is a root of some polynomial $f(x)\in\chi'[x]$ irreducible over $\chi'$. By induction, each coefficient of $f(x)$ is contained in a finite field, so there is a finite field $F$ containing all the coefficients of $f(x)$. If the degree of $f(x)$ is $n$, then $F(\alpha)$ will be a field in $\textbf{On}_p$ containing $\alpha$, and it will have order $|F|^n$, which is finite. So $F(\alpha)$ is the field we want.

\begin{lemma}
Let $\chi\in \textbf{On}_p$ be a field below the first transcendental. Then for any prime $u$, the following are equivalent:

1. For any $f(x)\in\chi[x]$ of degree $u$ such that $f(x)$ is irreducible over the field generated by its coefficients, $\chi$ contains the roots of $f(x)$.

2. $\chi$ contains finite fields of order $[p^{u^n}]$ for all $n\in\omega$.
\end{lemma}

\begin{proof}
$1\Rightarrow 2$: This can be proven by induction on $n$. Since $\chi$ is a field of characteristic $p$, $\chi$ obviously contains a field of order $[p^{u^0}]=p$. Assume that $F\subseteq\chi$ is a field of order $[p^{u^k}]$. Let $f(x)\in F[x]$ be a polynomial of degree $u$ that is irreducible over $F$; this polynomial has a root $\alpha\in\chi$, but then $F(\alpha)\subseteq\chi$ is a field of order $[p^{u^{k+1}}]$.

$2\Rightarrow 1$: Let $f(x)\in\chi[x]$ be a polynomial of degree $u$ that is irreducible over $F\subseteq\chi$, the finite field generated by the coefficients of $f(x)$. Say $|F|=[p^{mu^n}]$, where $m$ is not a multiple of $u$; then all of the roots of $f(x)$ are in the subfield of $\textbf{On}_p$ of order $[p^{mu^{n+1}}]$. But $\chi$ contains both a field of order $[p^{mu^n}]$ (namely, $F$) and a field of order $[p^{u^{n+1}}]$ (by assumption), so $\chi$ must contain a field of order $[p^{mu^{n+1}}]$. So $\chi$ contains the roots of $f(x)$.
\end{proof}

\begin{lemma}
Let $\chi\in \textbf{On}_p$ be a field below the first transcendental. Let $u_1,u_2,\ldots,u_m$ be primes. Then the following are equivalent:

1. For every $i$, the following holds: for any $f(x)\in\chi[x]$ of degree $u_i$ that is irreducible over the field generated by its coefficients, $\chi$ contains the roots of $f(x)$.

2. If all the prime factors of $n\in\NN$ are among the $u_i$, then the following holds: for any $f(x)\in\chi[x]$ of degree $n$ that is irreducible over the field generated by its coefficients, $\chi$ contains the roots of $f(x)$.
\end{lemma}

\begin{proof}
$2\Rightarrow 1$: trivial, as we can just take $n=u_i$ for each $u_i$ in turn.

$1\Rightarrow 2$: let $f(x)\in\chi[x]$ be a polynomial of degree $n$ that is irreducible over $F\subseteq\chi$, the finite field generated by its coefficients. Then all the roots of $f(x)$ are in the finite field of order $|F|^n$. But this field can be built up from $F$ using extensions of prime degree (i.e. of degree $u_i$ for some $i$), and by assumption, the fields resulting from each extension will be contained in $\chi$. So the roots of $f(x)$ will be in $\chi$.
\end{proof}

Essentially, this proves parts 1 through 3 of Theorem \ref{closureofp}. Below the first transcendental, the first fields will define quadratic extensions of themselves. If $\phi$ is a field, and the next field $\widetilde{\phi}$ is a quadratic extension, then we have seen that $\widetilde{\phi}=[\phi^2]$. So the first fields are $p$, $[p^2]$, $[p^4]$, and so on. The supremum of these fields ($[p^\omega]=\omega$) is quadratically closed, since it contains fields of order $[p^{2^n}]$ for all $n$. The subsequent fields $[\omega^3]$, $[\omega^9]$, etc. will each be cubic extensions of the previous field; there will be no need for further quadratic extensions, since these fields will still contain fields of order $[p^{2^n}]$ for all $n$. Once we have a cubically closed field, then come the quintic extensions (with no need for more cubic extensions), and so on. And the supremum of all these fields is the first transcendental, $[\omega^{\omega^\omega}]$.

To prove the next parts of Theorem \ref{closureofp}, we will introduce some notation from \cite{Lenstra2}. For $\alpha\in[\omega^{\omega^\omega}]$, we will let $d(\alpha)$ be the degree of the minimal polynomial of $\alpha$ over the field $p$. (So, the smallest field containing $\alpha$ has order $[p^{d(\alpha)}]$.) Also, if $\alpha\ne 0$, we will let $\text{ord}(\alpha)$ be the multiplicative order of $\alpha$: i.e. the smallest $n\in\NN$ where $\alpha^n=1$.

\begin{lemma}
For every nonzero $\alpha\in[\omega^{\omega^\omega}]$, $d(\alpha)$ is the smallest $n\in\NN$ where $\text{ord}(\alpha)$ divides $[p^n-1]$.
\end{lemma}

\begin{proof}
Let $n=d(\alpha)$. Then $\alpha$ is contained in a field of order $[p^n]$, but not contained in a field of order $[p^m]$ for any $m<n$. So $\alpha$ is a root in $\textbf{On}_p$ of $x^{[p^n]}-x$ (hence of $x^{[p^n-1]}-1$), but not of $x^{[p^m]}-x$ (nor of $x^{[p^m-1]}-1$) for any $m<n$. So the multiplicative order of $\alpha$ is a factor of $[p^n-1]$, but not of $[p^m-1]$ for any $m<n$.
\end{proof}

From here on, we will say that $r\in\NN$ is a ``primitive divisor'' of $[p^n-1]$ (for $n\in\NN$) if $r$ divides $[p^n-1]$, but $r$ does not divide $[p^m-1]$ for any $m\in\NN$, $m<n$. (Any factor of $[p^1-1]$ is automatically a primitive divisor of $[p^1-1]$.) Thus, we have proven that for every nonzero $\alpha\in[\omega^{\omega^\omega}]$, $\text{ord}(\alpha)$ is a primitive divisor of $[p^{d(\alpha)}-1]$.

Also, if $u,m,n\in \NN$, we will write $[u^m] \ \| \ n$ if $[u^m]$ divides $n$, but $[u^{m+1}]$ does not divide $n$. (If $u$ does not divide $n$, then we will write $[u^0] \ \| \ n$.)

\begin{lemma}
Say $F\subseteq \textbf{On}_p$ is a field of order $[p^n]$, $\alpha\in F$, and $u$ is a prime. Assume $[u^m] \ \| \ [p^n-1]$. Then $\alpha$ is a $u$th power in $F$ (i.e. there exists a $\beta\in F$ where $\beta^u=\alpha$) if and only if one of the following occurs:

1. $m=0$ (so $[p^n-1]$ is not a multiple of $u$)

2. $[u^m]$ does not divide $\text{ord}(\alpha)$.
\end{lemma}

\begin{proof}
If $m=0$, then there exists $k\in \NN$ such that $[ku]\equiv 1$ (mod $[p^n-1]$). Let $\beta=\alpha^k\in F$; then $\beta^u=\alpha^{[ku]}=\alpha$.

If $m>0$, then assume $\text{ord}(\alpha)=[u^ks]$, where $u$ is not a factor of $s$. Then $k\le m$, since $\text{ord}(\alpha)$ divides $[p^n-1]$, and $u^m$ is the largest power of $u$ dividing $[p^n-1]$. Let $\beta\in \textbf{On}_p$ be a $u$th root of $\alpha$; then $\text{ord}(\beta) \ | \ [u^{k+1}s]$. (We'll actually have $\text{ord}(\beta)=[u^{k+1}s]$ if $k\ge 1$; if $k=0$, $\text{ord}(\beta)$ may either be $s$ or $us$.) If $k<m$ (so that $[u^m]$ does not divide $\text{ord}(\alpha)$), then $\text{ord}(\beta) \ | \ [p^n-1]$, so $\beta\in F$. If $k=m$, then since $[u^m]$ divides $\text{ord}(\alpha)$, $[u^{m+1}]$ must divide $\text{ord}(\beta)$. So $\text{ord}(\beta)$ does not divide $[p^n-1]$, so $\beta\not\in F$.
\end{proof}

This shows that, below the first transcendental, no $p$th degree extensions in $\textbf{On}_p$ will be by $p$th roots. The reason: if $\alpha\in[\omega^{\omega^\omega}]$ is contained in a field of order $[p^n]$, then $\alpha$ is already a $p$th power in that finite field, since $[p^n-1]$ is not a multiple of $p$. But as we will see, if $u$ is a prime other than $p$, then all extensions of degree $u$ below the first trascendental will be by $u$th roots.

At this point, we will need a number-theoretic lemma before proving parts 4 through 6 of Theorem \ref{closureofp}.

\begin{lemma}
If $u$ is a prime, and $s$, $n$ are natural numbers ($s\ge 2$) such that $u \ | \ [s-1]$, then with one exception, $u^a \ \| \ [s^n-1]$ iff $u^a \ \| \ [n(s-1)]$. The exception: if $u=2$, $n$ is even, and $s\equiv 3$ (mod 4), then $u^a \ \| \ [s^n-1]$ iff $u^a \ \| \ [n(s+1)]$.
\end{lemma}

Among other things, this means that if $u$ divides $[s-1]$ but does not divide $n$, then the $u$-part of $[s-1]$ equals the $u$-part of $[s^n-1]$.

\begin{proof}
We have $[s^n-1]=[(s-1)(s^{n-1}+s^{n-2}+\cdots+s+1)]$. Since $s\equiv 1$ (mod $u$), we have $[s^{n-1}+s^{n-2}+\cdots+s+1]\equiv n$ (mod $u$). So if $n$ is not a multiple of $u$, then neither will
$[s^{n-1}+s^{n-2}+\cdots+s+1]$ be, and the $u$-part of $[s^n-1]$ will equal the $u$-part of $[s-1]$.

We will now show that, as long as $u>2$, the $u$-part of $[s^u-1]$ equals the $u$-part of $[u(s-1)]$. That will suffice to prove the theorem, except for the case where $u=2$; by repeatedly pulling off factors of $u$, we find that the $u$-part of $[s^{u^k}-1]$ equals the $u$-part of $[u^k(s-1)]$ for all positive integers $u$. Thus, if $n=mu^k$ (where $u$ does not divide $m$), then the $u$-part of $[s^n-1]$ equals the $u$-part of $[u^k(s^m-1)]$, which equals the $u$-part of $[u^k(s-1)]$, which equals the $u$-part of $[n(s-1)]$.

We have $[s^{u-1}+s^{u-2}+\cdots+s+1]=[u+(s-1)[s^{u-2}+2s^{u-3}+\cdots+(u-2)s+(u-1)]]$. Since $s\equiv 1$ (mod $u$), we have $[s^{u-2}+2s^{u-3}+\cdots+(u-2)s+(u-1)]\equiv[1+2+\cdots+(u-1)]=[\frac{u(u-1)}{2}]\equiv 0$ (mod $u$). Thus, $[s^{u-1}+s^{u-2}+\cdots+s+1]$ is of the form $[u+u^2m]$ for some positive integer $m$; the $u$-part of $[s^{u-1}+s^{u-2}+\cdots+s+1]$ must be exactly $u$. Thus, the $u$-part of $[s^u-1]=[(s-1)(s^{u-1}+s^{u-2}+\cdots+s+1)]$ must equal the $u$-part of $[u(s-1)]$.

What remains is the case where $u=2$. In that case, the $2$-part of $[s^2-1]=[(s-1)(s+1)]$ equals the $2$-part of $[2(s-1)]$ if $s\equiv 1$ (mod 4), but it equals the $2$-part of $[2(s+1)]$ if $s\equiv 3$ mod 4. Since $[s^m]\equiv 1$ (mod 4) whenever $m$ is even, we can conclude that the $2$-part of $[s^4-1]$ equals the $2$-part of $[2(s^2-1)]$, the $2$-part of $[s^8-1]$ equals the $2$-part of $[4(s^2-1)]$, and so on. In general, the $2$-part of $[s^n-1]$ (for $n$ even) always equals the $2$-part of $[\frac{n}{2}(s^2-1)]$, which equals the $2$-part of $[n(s-1)]$ (for $s\equiv 1$ (mod 4)) or the $2$-part of $[n(s+1)]$ (for $s\equiv 3$ (mod 4)).
\end{proof}

We can now prove part 4 of Theorem \ref{closureofp}.

\begin{lemma}
For each prime $u\ne p$, there exists an element in $\chi_u$ that has no $u$th root in $\chi_u$. Thus, if $\alpha_u$ is the smallest such element, then $\chi_u^u=\alpha_u$.
\end{lemma}

\begin{proof} Assume that $u$ is a primitive divisor of $[p^m-1]$; assume $[u^n] \ \| \ [p^m-1]$. We have $m\le u-1$, since $u$ must divide $[p^{u-1}-1]$. Thus, all prime factors of $m$ are less than $u$. Since $\chi_u$ is closed under extensions of degree less than $u$, $\chi_u$ must contain a finite field of order $[p^m]$, and thus must contain an element of multiplicative order $[u^n]$.

However, $\chi_u$ cannot contain any elements of multiplicative order $[u^{n+1}]$, for the following reason: for any $\alpha\in\chi_u$, let $F\subseteq\chi_u$ be a finite field containing $\alpha$. Then $|F|=[p^s]$, where all prime factors of $s$ are less than $u$. Then $u \ | \ [p^s-1]$ if and only if $s$ is a multiple of $m$; if that is true, then since $u$ does not divide $s$, the $u$-part of $[p^s-1]$ is the
same as the $u$-part of $[p^m-1]$ (namely, $[u^n]$). It is thus impossible for $[u^{n+1}]$ to divide $[p^s-1]$, so no element of $F$ (including $\alpha$) can have multiplicative order $[u^{n+1}]$.

Thus, if $\alpha\in\chi_u$ has multiplicative order $[u^n]$ (and there must be such an element in $\chi_u$), then $\alpha$ has no $u$th root in $\chi_u$; any $u$th root of $\alpha$ would have multiplicative order $[u^{n+1}]$, and thus cannot be in $\chi_u$.
\end{proof}

Now to prove parts 5 and 6 of Theorem \ref{closureofp}.

\begin{lemma}
We have $(\chi_{[u^{n+1}]})^u=\chi_{[u^n]}$ whenever $p\ne u$ and $n\ge 1$, with one exception: we have $(\chi_4)^2=\chi_2+1$ when $p\equiv 3$ (mod $4$).
\end{lemma}

\begin{proof}
Assume that we are not in the exceptional case: either $u>2$, or $u=2$ and $p\equiv 1$ (mod $4$). By induction, we will assume that $(\chi_{[u^{i+1}]})^u=\chi_{[u^i]}$ whenever $1\le i<n$, and we will prove that $(\chi_{[u^{n+1}]})^u=\chi_{[u^n]}$. Note: since $(\chi_{[u^n]})^{[u^n]}=\alpha_u$, if $\text{ord}(\alpha_u)=a$, then $\text{ord}(\chi_{[u^n]})=[au^n]$. And if $d(\alpha_u)=k$, then $d(\chi_{[u^n]})=[ku^n]$.

Consider any $\beta<\chi_{[u^n]}$; we claim that $\beta$ is a $u$th power in $\chi_{[u^{n+1}]}$. If $\beta$ is already a $u$th power in its minimal field $p(\beta)$, then we're done. Otherwise, if $|p(\beta)|=[p^b]$ (then $d(\beta)=b$), and if $[u^c] \ \| \ [p^b-1]$, then $[u^c] \ \| \ \text{ord}(\beta)$. (Note that $c>0$; if $u$ did not divide $[p^b-1]$, then $\beta$ would be a $u$th power in $p(\beta)$.) But then the finite field $p(\beta,\chi_{[u^n]})$ has dimension a multiple of $[bu]$ (say, $[bju]$). And since $[u^{c+1}] \ | \ [p^{[bju]}-1]$, but $[u^{c+1}]$ does not divide $\text{ord}(\beta)$, $\beta$ must be a $u$th power in $p(\beta,\chi_{[u^n]})$. So all ordinals less than $\chi_{[u^n]}$ are $u$th powers in $\chi_{[u^{n+1}]}$.

But $\chi_{[u^n]}$ is not a $u$th power in $\chi_{[u^{n+1}]}$, for the following reasons: if $d(\alpha_u)=k$, and if $[u^m] \ \| \ [p^k-1]$, then $[u^m] \ \| \ \text{ord}(\alpha_u)$. So, $[u^{m+n}] \ \| \ \text{ord}(\chi_{[u^n]})$. Consider any finite field $F\subseteq\chi_{[u^{n+1}]}$ containing $\chi_{[u^n]}$; it has dimension $u^nt$, where $t$ is a multiple of $k$, and all prime factors of $t$ are less than $u$. We have $[u^{m+n}] \ \| \ [p^{u^nt}-1]$; since $[u^{m+n}] \ \| \ \text{ord}(\chi_{[u^n]})$ also, $\chi_{[u^n]}$ is not a $u$th power in any finite field in $\chi_{[u^{n+1}]}$ (hence, it is not a $u$th power in $\chi_{[u^{n+1}]}$). So $\chi_{[u^n]}$ is the smallest element of $\chi_{[u^{n+1}]}$ that has no $u$th root in $\chi_{[u^{n+1}]}$; we must have $(\chi_{[u^{n+1}]})^u=\chi_{[u^n]}$.

Now assume $u=2$ and $p\equiv 3$ mod $4$. The same reasoning as before shows that if $\beta<\chi_2=p$, then $\beta$ is a square in $\chi_4=[p^2]$. But $\chi_2$ will also be a square in $\chi_4$: since $p\equiv 3$ mod $4$, we have $[2^1] \ \| \ [p-1]$. So $[2^1] \ \| \ \text{ord}(\alpha_2)$, and $[2^2] \ \| \ \text{ord}(\chi_2)$. But $[2^3] \ | \ [p^2-1]$. So the largest [power] of $2$ dividing $[p^2-1]$ does not divide the order of $\chi_2$, so $\chi_2$ is a square in the field of order $[p^2]$ (namely, $\chi_4$).

So if $p\equiv 3$ mod $4$, then $\chi_2$ is a square in $\chi_4$. If $\chi_2+1$ is a square in $\chi_4$, then for some $a,b\in p$, we have $\chi_2+1=(a\chi_2+b)^2=(2\cdot ab)\chi_2+(a^2\alpha_2+b^2)$. So we have $a^2\alpha_2+b^2=2\cdot ab$ (both sides of that equation are equal to $1$), and by manipulating that equation, we obtain $(b-a)^2=a^2(1-\alpha_2)$. Since $(b-a)^2$ and $a^2$ are squares in $\chi_2$, $1-\alpha_2$ must also be a square in $\chi_2$.

But it is not possible for $1-\alpha_2$ to be a square in $\chi_2$, for the following reasons: since $\alpha_2$ is the smallest non-square in $\chi_2$, $[\alpha_2-1]=\alpha_2-1$ is a square in $\chi_2$. And since $p\equiv 3$ (mod $4$), $-1$ is not a square in $\chi_2$. So $1-\alpha_2=(-1)(\alpha_2-1)$, the product of a non-square and a square in $\chi_2$, cannot be a square in $\chi_2$. So there is no element $a\chi_2+b\in\chi_4$ whose square is $\chi_2+1$; we must have $\chi_4^2=\chi_2+1$.

However, the inductive step works thereafter: if $[2^{m'}] \ \| \ \text{ord}(\chi_2+1)$, then $[2^{m'}] \ \| \ [p^2-1]$, so $[2^{m'+1}] \ \| \ [p^4-1]$. And since $[2^{m'+1}] \ \| \ \text{ord}(\chi_4)$, $\chi_4$ will not be a square in $\chi_8$. So $\chi_8^2=\chi_4$, and we may proceed by induction as before. So $(\chi_{[u^{n+1}]})^u=\chi_{[u^n]}$ in all cases, except that $(\chi_4)^2=\chi_2+1$ when $p\equiv 3$ mod $4$.
\end{proof}

It remains to prove part 7 of Theorem \ref{closureofp}. For simplicity, let $\chi=\chi_p$; then $\chi_{[p^n]}=[\chi^{p^{n-1}}]$ for all $n\in\NN$. As discussed, each of these fields is a degree $p$ extension of the previous field, but not by a $p$th root (since all elements of finite fields of characteristic $p$ already have $p$th roots in that finite field). So each $\chi_{[p^n]}$ will be an extension of $\chi_{[p^{n-1}]}$ by a root of a polynomial of form $x^p-x-\alpha$, assuming there is such a polynomial with no roots in $\chi_{[p^n]}$. And as we will see, there will always be such a polynomial; the smallest suitable $\alpha$ will be $\displaystyle\prod_{k=1}^{n-1}(\chi_{[p^k]})^{[p-1]}=[\chi^{p^n-1}]$.

Let $f(x)=x^p-x$; for any field $F\subseteq \textbf{On}_p$, $f$ is an additive homomorphism from $F$ to itself. Let $S$ be the set of all ordinals in $\chi$ that are [multiples] of $p$; that is, all ordinals of the form $[p\delta]$ for some ordinal $\delta$.

\begin{lemma}
The map $f(x)$ sends the ordinals in $\chi$ to exactly the ordinals in $S$.
\end{lemma}

\begin{proof}
First, we'll show that only ordinals in $S$ get mapped to. Since $f$ is an additive homomorphism, we only need to show that the elements of $\chi$ that are \textbf{groups} get sent to elements of $S$.

Let $\alpha$ be an arbitrary group in $\chi$; we can assume $\alpha>1$, since $f(1)=0\in S$. Then $\alpha$ has the form $\chi_r^m\delta$, where $r=[u^n]$ is a [power] of a prime $u<p$, $m$ is a positive integer less than $u$, and $\delta\in\chi_r$ is a group.

Then $f(\alpha)=f(\chi_r^m\delta)=\chi_r^{[mp]}\delta^p-\chi_r^m\delta=\chi_r^{[au+b]}\delta^p-\chi_r^m\delta$, where $a,b$ are nonnegative integers, and $b<p$. Since $[mp]$ is not a multiple of $u$, $b>0$. We then have $\chi_r^{[au+b]}\delta^p-\chi_r^m\delta=\chi_r^b(\chi_r^u)^a\delta^p-\chi_r^m\delta=\chi_r^b(\delta')-\chi_r^m\delta$, where $\delta'=(\chi_r^u)^a\delta^p\in\chi_r$. Since $b,m>0$, both $\chi_r^b(\delta')$ and $\chi_r^m\delta$ are in $S$; thus, so is $\chi_r^b(\delta')-\chi_r^m\delta$. So all groups in $\chi$ (and hence, all elements of $\chi$) are mapped by $f$ into $S$.

It remains to show that $f$ maps $\chi$ to the entire set $S$. Choose an arbitrary $\alpha\in S$, and let $F=p(\alpha)$, the smallest finite field containing $\alpha$. Then $f$ is an additive homomorphism from $F$ to itself, and its kernel contains $p$ elements (namely, the ordinals less than $p$). So the image of $f$ on $F$ contains $[|F|/p]$ elements. But the image must contain only elements of
$F\cap S$, and there are exactly $[|F|/p]$ elements in $F\cap S$. So all elements of $F\cap S$, including $\alpha$ itself, must be mapped to by some element of $F$. So all elements of $S$ are mapped to by some element in $\chi$.
\end{proof}

This is enough to establish the following:

\begin{corollary}
$\chi$ is a root of $x^p-x-1$.
\end{corollary}

The reason: all polynomials in $\chi[x]$ of form $x^p-\alpha$ already have roots in $\chi$. So does the polynomial $x^p-x$ (namely, all the ordinals less than $p$). But $x^p-x-1$ has no root in $\chi$; given any $\alpha\in\chi$, $f(\alpha)=\alpha^p-\alpha$ is a [multiple] of $p$, hence cannot be 1. So $\chi$ must be a root of $x^p-x-1$.

The rest of part 7 of Theorem \ref{closureofp} will be established by the following theorem:

\begin{theorem}
For each field $\chi_{[p^{n+1}]}=[\chi^{p^n}]$ (for $n\in\omega$), the image of $f$ on $[\chi^{p^n}]$ is the set $S_n=\{\sum_{k=0}^{[p^n-1]}[\chi^k]\beta_k:\beta_k\in\chi,\beta_{[p^n-1]}\in S\}$. The smallest element of $[\chi^{p^n}]$ that is not in $S_n$ is $[\chi^{p^n-1}]$; thus $[\chi^{p^n}]$ is a root of $x^p-x-[\chi^{p^n-1}]$.
\end{theorem}

For example, in $\textbf{On}_3$, we have $\chi=\chi_3=\omega$. All elements of $\chi_{27}=[\omega^9]$ have the form $[\omega^8]\beta_8+\cdots+\omega\beta_1+\beta_0$, where each $\beta_i\in \omega$; the elements that are in the image of $f$ on $\chi_{27}$ are exactly those for which $\beta_8\in S$. The smallest element of $\chi_{27}$ not of that form is $[\omega^8]$; thus, $\chi_{27}$ is a root of $x^3-x-[\omega^8]$.

\begin{proof}
We will prove this by induction. We already established the $n=0$ case, so we will assume it is true for $n-1$, and prove it is true for $n$.

For simplicity, let $\phi=[\chi^{p^{n-1}}]$, and let $\widetilde{\phi}$ be the next field, $[\chi^{p^n}]$. Let $\gamma\in\widetilde{\phi}$; we have $\displaystyle\gamma=\sum_{k=0}^{[p-1]}\phi^k\gamma_k$, where each $\gamma_k\in\phi$. Then $$f(\gamma)=\sum_{k=0}^{[p-1]}f(\phi^k\gamma_k)=\sum_{k=0}^{[p-1]}(\phi^{[kp]}\gamma_k^p-\phi^k\gamma_k)$$ $$=\sum_{k=0}^{[p-1]}([\phi+\chi^{p^{n-1}-1}]^k\gamma_k^p-\phi^k\gamma_k)=\phi^{[p-1]}(\gamma_{[p-1]}^p-\gamma_{[p-1]})+\delta,$$ where $\delta\in\phi^{[p-1]}$. By the inductive hypothesis,
$$(\gamma_{[p-1]}^p-\gamma_{[p-1]})=f(\gamma_{[p-1]})=\sum_{k=0}^{[p^{n-1}-1]}[\chi^k]\beta_k,$$ where
$\beta_k\in\chi,\beta_{[p^{n-1}-1]}\in S$. Thus, $f(\gamma)$ is equal to $[\chi^{p^n-1}]\beta_{[p^{n-1}-1]}$ plus a sum of ``lesser terms''; $f(\gamma)\in S_n$.

To show that the image of $f$ on $[\chi^{p^n}]$ includes the entire set $S_n$: let $\gamma\in S_n$, and let
$\displaystyle\gamma=\sum_{k=0}^{[p^n-1]}[\chi^k]\gamma_k$, where each $\gamma_k\in\chi$. (Then $\gamma_{[p^n-1]}\in S$.) Let $F$ be the smallest finite field containing each $\gamma_k$ and each $[\chi^k]$ for all $k$ from $0$ to $[p^n-1]$; then $\gamma\in F$. The mapping $f$ is an additive homomorphism from $F$ to itself, and its kernel contains $p$ elements (the ordinals less than $p$). So the image of $f$ over $F$ must contain $[|F|/p]$ elements. But the image is contained in $F\cap S_n$, and there are exactly $[|F|/p]$ elements in $F\cap S_n$. So the image of $f$ over $F$ is exactly the set $F\cap S_n$, which contains $\gamma$. So all elements $\gamma\in S_n$ are mapped to by some element of $[\chi^{p^n}]$.
\end{proof}

That completes the proof of part 7 of Theorem \ref{closureofp}: for each $n$, $[\chi^{p^n}]$ is a root of $x^p-x-[\chi^{p^n-1}]$, since $[\chi^{p^n-1}]$ is the smallest element of $[\chi^{p^n}]$ that is not mapped to by $f(x)$.

In summary, we have now determined the full structure of $[\omega^{\omega^\omega}]\in \textbf{On}_p$, with the exception of the unknown elements $\alpha_u\in \textbf{On}_p$ for each prime $u\ne p$. (We will discuss how to find each $\alpha_u$ in the next section.) The structure is similar for all primes $p$; $[\omega^{\omega^\omega}]$ is always the first transcendental, and if $\phi\in \textbf{On}_2$ is an infinite field below $[\omega^{\omega^\omega}]$, then $\phi$ is a field in every $\textbf{On}_p$. I would conjecture that this pattern continues beyond $[\omega^{\omega^\omega}]$:

\begin{conje}
If $\phi\in \textbf{On}_2$ is an infinite field, then $\phi$ is a field in $\textbf{On}_p$ for all primes $p$. If $\phi\in \textbf{On}_2$ is a transcendental, then $\phi$ is a transcendental in $\textbf{On}_p$ for all primes $p$.
\end{conje}

But a proof of this conjecture would seem to be well out of reach. Only one transcendental in $\textbf{On}_2$ is currently known: namely, $[\omega^{\omega^\omega}]$. The problem of finding the second transcendental in $\textbf{On}_2$ is wide open; the problem would seem to be equally difficult in $\textbf{On}_p$ for other primes $p$.

\section{Effective computation below the first transcendental in $\textbf{On}_p$}

We will now discuss methods for finding the elements $\alpha_u\in \textbf{On}_p$ for each prime $u\ne p$. In \cite{Lenstra2}, it is shown that the elements $\alpha_u$ can be effectively determined in $\textbf{On}_2$; we will show that the same can be done in $\textbf{On}_p$ for every prime $p$. So multiplication can be done effectively in $[\omega^{\omega^\omega}]$; division can be done effectively as well, since all elements of $[\omega^{\omega^\omega}]$ have finite multiplicative order.

As before, for $\alpha\in[\omega^{\omega^\omega}]$, let $d(\alpha)$ be the degree of the irreducible polynomial of $\alpha$ over the field $p$. Based on our results from Section 3, we have the following proposition (which is identical to Proposition 1.8 in \cite{Lenstra2}):

\begin{proposition}\label{degreeofchi}
We have that $\chi_r$ (for any prime [power] $r$) is the smallest element of $[\omega^{\omega^\omega}]$ whose irreducible polynomial has degree divisible by $r$. In other words, if $r=[u^n]$ (for $u$ prime), then $\chi_r$ is the set of all ordinals $\alpha\in[\omega^{\omega^\omega}]$ where $d(\alpha)$ is divisible only by primes $\le u$ and where $r$ does not divide $d(\alpha)$.
\end{proposition}

Following \cite{Lenstra2}, we will extend this notation and define $\chi_h$ for all positive integers $h$; $\chi _h$ is the smallest ordinal $\alpha\in[\omega^{\omega^\omega}]$ where $d(\alpha)$ is divisible by $h$. (We clearly have $\chi _1=0$.)

The following lemma is a generalization of Lemma 2.5 in \cite{Lenstra2}.

\begin{lemma}\label{degreeofsums}
Let $\beta,\gamma\in \textbf{On}_p$ be elements of $[\omega^{\omega^\omega}]$, and let $r=[u^n]$ be a [power] of a prime $u$. If $r$ divides $d(\beta)$ but not $d(\gamma)$, then $r$ divides both $d(\beta+\gamma)$ and $d(\beta-\gamma)$.
\end{lemma}

\begin{proof}
We have $\beta\in p(\gamma,\beta+\gamma)$, which is an extension of $p$ of degree $\text{lcm}(d(\gamma),d(\beta+\gamma))$. So $d(\beta)$ divides $\text{lcm}(d(\gamma),d(\beta+\gamma))$.

Since $r$ divides $d(\beta)$, $r$ divides $\text{lcm}(d(\gamma),d(\beta+\gamma))$; since $r$ does not divide $d(\gamma)$, $r$ must divide $d(\beta+\gamma)$.

Similar reasoning shows that $r$ divides $d(\beta-\gamma)$.
\end{proof}

Lemma \ref{degreeofsums} will be used in proving the following two results; they are generalizations to $\textbf{On}_p$ of Theorem 2.1 and Corollary 2.2 in \cite{Lenstra2}.

\begin{theorem}\label{chisums}
Let $h>1$ be a natural number, let $u$ be the smallest prime dividing $h$, let $r$ be the largest [power] of $u$ dividing $h$, and let $g=[h/r]$. Then $\chi_h=\chi_g$ if $r$ divides $d(\chi_g)$; otherwise, $\chi_h=\chi_g+\chi_r=[\chi_g+\chi_r]$.
\end{theorem}

\begin{proof}
The theorem is proven by induction on the number of primes dividing $h$. If $h=r$, then $g=1$; the theorem holds true in that case, since $\chi_h=0+\chi_r=\chi_g+\chi_r$, and $r$ does not divide $d(\chi_g)=d(0)=1$. So assume $h$ has at least two distinct prime divisors. Note that we must have $\chi_h\ge\chi_g$, since $g$ divides $h$.

By the inductive hypothesis, $\chi_g$ is a finite sum of terms $\chi_{r'}$, where each $r'$ is a [power] of a larger prime than $u$. Thus, each $\chi_{r'}$ is larger than $\chi_r$. We can thus conclude that, for all $\alpha\le\chi_r$, we have $\chi_g+\alpha=[\chi_g+\alpha]$. Specifically, we have $\chi_g+\chi_r=[\chi_g+\chi_r]$.

Assume that $r$ does divide $d(\chi_g)$. By definition, $g$ also divides $d(\chi_g)$; thus, $h$ divides $d(\chi_g)$. But $\chi_h$ is the smallest element of $[\omega^{\omega^\omega}]$ whose minimal polynomial has degree a multiple of $h$; thus, $\chi_g\ge\chi_h$. We already showed $\chi_h\ge\chi_g$, hence they are equal.

Now assume that $r$ does not divide $d(\chi_g)$. Since $r$ does divide $d(\chi_r)$, it follows from Lemma \ref{degreeofsums} that $r$ divides $d(\chi_g+\chi_r)$. Now, $g$ and $d(\chi_r)$ are relatively prime, since every prime dividing $d(\chi_r)$ is at most $u$ (based on the work in Section 3), but every prime dividing $g$ is greater than $u$. On the other hand, $g$ divides $d(\chi_g)$. So applying Lemma
\ref{degreeofsums} to each prime [power] factor of $g$, we get that $g$ must divide $d(\chi_g+\chi_r)$.

But if both $g$ and $r$ divide $d(\chi_g+\chi_r)$, then $h$ must divide $d(\chi_g+\chi_r)$. Thus, $\chi_h\le\chi_g+\chi_r=[\chi_g+\chi_r]$.

We've now proven that $\chi_g\le\chi_h\le[\chi_g+\chi_r]$; that is only possible if $\chi_h=[\chi_g+\alpha]$ for some $\alpha\le\chi_r$. Thus $\chi_h=\chi_g+\alpha$. But then $\alpha=\chi_h-\chi_g$, and since $r$ divides $d(\chi_h)$ but not $d(\chi_g)$, $r$ must divide $d(\alpha)$ by Lemma \ref{degreeofsums}. Since $r$ divides $d(\alpha)$, we have $\alpha\ge\chi_r$; thus, $\alpha=\chi_r$, and we're done.
\end{proof}

\begin{corollary}
For every natural number $h$, there is a unique finite set $Q(h)$ of prime [powers] where $\displaystyle\chi_h=\sum_{r\in Q(h)}\chi_r$. (When there is ambiguity about the choice of field $\textbf{On}_p$, we will write $Q_p(h)$ for $Q(h)$.) Every $r\in Q(h)$ divides $h$ and is relatively prime to $[h/r]$. Finally, if $h>1$ and $u$ is the largest prime dividing $h$, then the largest [power] of $u$ dividing $h$ belongs to $Q(h)$.
\end{corollary}

This corollary follows from Theorem \ref{chisums} by induction; we can rewrite $\chi_h$ as either $\chi_g$ or $\chi_g+\chi_r$, and if $g$ is not a prime [power], then we can rewrite $\chi_g$ in a similar fashion, and so on.

Next, a lemma generalizing Lemma 3.4 from \cite{Lenstra2}.

\begin{lemma}\label{fieldgenerator}
If $\chi\in[\omega^{\omega^\omega}]$, then the multiplicative group of the field $\omega(\chi)$ is generated by the elements $\chi+m$,
$m\in\omega$.
\end{lemma}

The proof is identical to the proof in \cite{Lenstra2}; the proof does not depend on the value of $p$.

\begin{proof}
Let $f(x)=\sum x^ia_i\in\omega[x]$ (where each $a_i\in \omega$) be the minimal polynomial in $\omega[x]$ with $\chi$ as a root. Let $\beta$ be any nonzero element of $\omega(\chi)$; assume $\beta=\sum \chi^jb_j$, where each $b_j\in \omega$. Let $\mu$ be the subfield of $\omega$ generated by all the coefficients $a_i$ and $b_j$. Then the polynomials $g(x)=\sum x^jb_j$ and $f(x)$ are both contained in $\mu[x]$. Since $f(x)$ is irreducible in $\mu[x]$, and since $g(x)$ (a nonzero polynomial) has smaller degree than $f(x)$, $g(x)$ and $f(x)$ must be relatively prime in $\mu[x]$.

We may then apply Kornblum-Artin's analogue of Dirichlet's theorem on primes in arithmetic progressions (which appears on pg. 94 of \cite{Artin} and on pg. 39 of \cite{Rosen}); if $t\in\omega$ is sufficiently large, then there exists a monic polynomial $h(x)\in\mu[x]$ of degree $t$ where $h(x)\equiv g(x)$ mod $f(x)$ (hence, $h(\chi)=\beta$). If $t$ is chosen to be a [power] of $2$, then since $\omega$ is quadratically closed, $h(x)$ is a product of linear factors in $\omega[x]$. If $\displaystyle h(x)=\prod_i(x+m_i)$ (where each $m_i\in\omega$), then $\displaystyle\beta=h(\chi)=\prod_i(\chi+m_i)$. So all nonzero elements of $\omega(\chi)$ are products of elements of form $\chi+m$ ($m\in\omega$), and that proves the lemma.
\end{proof}

We will use Lemma \ref{fieldgenerator} to prove the next theorem, a generalization to $\textbf{On}_p$ of Theorem 3.1 from \cite{Lenstra2}. It allows us to fully classify the elements $\alpha_u$ up to a finite term.

For any prime $u\ne p$, let $\zeta_u$ be a primitive $u$th root of unity in $[\omega^{\omega^\omega}]$. Let $f(u)=d(\zeta_u)$; then $u$ is a primitive divisor of $[p^{f(u)}-1]$. (Thus, $f(u)$ is a divisor of $[u-1]$.)

\begin{theorem}
For any prime number $u\ne p$, there exist natural numbers $m$ and $m'$ where $\alpha_u=[\chi_{f(u)}+m]=\chi_{f(u)}+m'$.
\end{theorem}

As in \cite{Lenstra2}, we will refer to $m$ as the ``excess'' of $\alpha_u$ over $\chi_{f(u)}$.

\begin{proof}
By definition, $\alpha_u$ is not a $u$th power in the field $\chi_u$. Let $F=p(\alpha_u)$, and let $F^*$ be the multiplicative group of the field $F$. Consider the additive homomorphism $\theta_1:F^*\rightarrow F^*$, where $\theta_1(a)=a^u$; $\alpha_u$ is not in the image of this map, so it is not surjective. So it is not injective either; there are elements of $F$ (other than 1) whose $u$th power is $1$, and the only such elements are the primitive $u$th roots of unity. Thus, $\zeta_u\in F=p(\alpha_u)$. So $d(\alpha_u)$ is a multiple of $d(\zeta_u)=f(u)$, so $\alpha_u\ge\chi_{f(u)}$.

On the other hand, since $d(\chi_{f(u)})$ is divisible by $f(u)$, $\zeta_u$ must be an element of $p(\chi_{f(u)})$. Let $G=p(\chi_{f(u)})$, and let $G^*$ be the multiplicative group of the field $G$. Consider the additive homomorphism $\theta_2:G^*\rightarrow G^*$, where $\theta_2(a)=a^u$; since $\zeta_u\in G^*$, this map cannot be injective. So it is not surjective either; there is some $\beta\in G$ that is
not a $u$th power in $G$. Now, we have $G=p(\chi_{f(u)})\subseteq\chi_u$, and all extensions of $p(\chi_{f(u)})$ contained within $\chi_u$ are of degree less than $u$. So, $\beta$ is not a $u$th power in $\chi_u$.

Note that $\beta\in\omega(\chi_{f(u)})$; by Lemma \ref{fieldgenerator}, we can write $\beta$ as a product of elements of the form $\chi_{f(u)}+m$, $m\in\omega$. Since $\beta$ is not a $u$th power in $\chi_u$, at least one of the factors $\chi_{f(u)}+m_0$ is also not a $u$th power in $\chi_u$. Thus, $\alpha_u\le\chi_{f(u)}+m_0$ (since $\alpha_u$ is the smallest element of $\chi_u$ that is not a $u$th power in $\chi_u$).

Finally, write $\chi_{f(u)}$ as $[\lambda+m_1]$, where $\lambda$ is a limit ordinal and $m_1\in\omega$. So, $\alpha_u\ge[\lambda+m_1]$. We have $\lambda+m=[\lambda+m]$ for all $m\in\omega$, so $\alpha_u\le\chi_{f(u)}+m_0=\lambda+m_1+m_0=[\lambda+m_2]$, where $m_2=m_1+m_0$. So we have $[\lambda+m_1]\le\alpha_u\le[\lambda+m_2]$ for some natural numbers $m_1$ and $m_2$; $\alpha_u=[\lambda+m_1+m]$ for some natural
number $m$.

From that, we can conclude the following: $\alpha_u=[(\lambda+m_1)+m]=[\chi_{f(u)}+m]$, and
$\alpha_u=[\lambda+(m_1+m)]=\lambda+[m_1+m]=\chi_{f(u)}-m_1+[m_1+m]=\chi_{f(u)}+m'$, for some natural number $m'$. That completes the proof.
\end{proof}

Using these theorems, and with the aid of Mathematica, I was able to assemble Tables \ref{On2} through \ref{On11}, which contain the elements $\alpha_u\in \textbf{On}_p$ for $u\le 43$, $p\le 11$. (The table for $\textbf{On}_2$ first appeared in \cite{Lenstra2}.)

\begin{table}
\caption[]{$\alpha_u\in \textbf{On}_2$} \label{On2}
\begin{tabular}{|c|c|c|c|c|}
\hline $u$ & $f(u)$ & $Q(f(u))$ & excess & $\alpha_u$ \\
\hline 3 & 2 & \{2\} & 0 & 2 \\
5 & 4 & \{4\} & 0 & $[2^2]$ \\
7 & 3 & \{3\} & 1 & $[2^\omega]+1$ \\
11 & 10 & \{5\} & 1 & $[2^{\omega^2}]+1$ \\
13 & 12 & \{3,4\} & 0 & $[2^\omega]+[2^2]$ \\
17 & 8 & \{8\} & 0 & $[2^4]$ \\
19 & 18 & \{9\} & 4 & $[2^{\omega\cdot\ 3}]+4$ \\
23 & 11 & \{11\} & 1 & $[2^{\omega^4}]+1$ \\
29 & 28 & \{7,4\} & 0 & $[2^{\omega^3}]+[2^2]$ \\
31 & 5 & \{5\} & 1 & $[2^{\omega^2}]+1$ \\
37 & 36 & \{9,4\} & 0 & $[2^{\omega\cdot\ 3}]+[2^2]$ \\
41 & 20 & \{5\} & 1 & $[2^{\omega^2}]+1$ \\
43 & 14 & \{7\} & 1 & $[2^{\omega^3}]+1$ \\
\hline
\end{tabular}
\end{table}

\begin{table}
\caption[]{$\alpha_u\in \textbf{On}_3$} \label{On3}
\begin{tabular}{|c|c|c|c|c|}
\hline $u$ & $f(u)$ & $Q(f(u))$ & excess & $\alpha_u$ \\
\hline 2 & 1 & $\emptyset$ & 2 & 2 \\
5 & 4 & \{4\} & 1 & $[3^2]+1$ \\
7 & 6 & \{3,2\} & 0 & $[3^\omega]+3$ \\
11 & 5 & \{5\} & 1 & $[3^{\omega^2}]+1$ \\
13 & 3 & \{3\} & 0 & $[3^\omega]$ \\
17 & 16 & \{16\} & 1 & $[3^8]+1$ \\
19 & 18 & \{9,2\} & 0 & $[3^{\omega\cdot\ 3}]+3$ \\
23 & 11 & \{11\} & 1 & $[3^{\omega^4}]+1$ \\
29 & 28 & \{7,4\} & 0 & $[3^{\omega^3}]+[3^2]$ \\
31 & 30 & \{5,3\} & 0 & $[3^{\omega^2}]+[3^\omega]$ \\
37 & 18 & \{9,2\} & 0 & $[3^{\omega\cdot\ 3}]+3$ \\
41 & 8 & \{8\} & 1 & $[3^4]+1$ \\
43 & 42 & \{7\} & 1 & $[3^{\omega^3}]+1$ \\
\hline
\end{tabular}
\end{table}

\begin{table}
\caption[]{$\alpha_u\in \textbf{On}_5$} \label{On5}
\begin{tabular}{|c|c|c|c|c|}
\hline $u$ & $f(u)$ & $Q(f(u))$ & excess & $\alpha_u$ \\
\hline 2 & 1 & $\emptyset$ & 2 & 2 \\
3 & 2 & \{2\} & 1 & $5+1$ \\
7 & 6 & \{3\} & 1 & $[5^\omega]+1$ \\
11 & 5 & \{5\} & 0 & $[5^{\omega^2}]$ \\
13 & 4 & \{4\} & 1 & $[5^2]+1$ \\
17 & 16 & \{16\} & 1 & $[5^8]+1$ \\
19 & 9 & \{9\} & 1 & $[5^{\omega\cdot\ 3}]+1$ \\
23 & 22 & \{11,2\} & 0 & $[5^{\omega^4}]+5$ \\
29 & 14 & \{7\} & 1 & $[5^{\omega^3}]+1$ \\
31 & 3 & \{3\} & 1 & $[5^\omega]+1$ \\
37 & 36 & \{9,4\} & 0 & $[5^{\omega\cdot\ 3}]+[5^2]$ \\
41 & 20 & \{5,4\} & 0 & $[5^{\omega^2}]+[5^2]$ \\
43 & 42 & \{7\} & 1 & $[5^{\omega^3}]+1$ \\
\hline
\end{tabular}
\end{table}

\begin{table}
\caption[]{$\alpha_u\in \textbf{On}_7$} \label{On7}
\begin{tabular}{|c|c|c|c|c|}
\hline $u$ & $f(u)$ & $Q(f(u))$ & excess & $\alpha_u$ \\
\hline 2 & 1 & $\emptyset$ & 3 & 3 \\
3 & 1 & $\emptyset$ & 2 & 2 \\
5 & 4 & \{4\} & 1 & $[7^2]+1$ \\
11 & 10 & \{5\} & 1 & $[7^{\omega^2}]+1$ \\
13 & 12 & \{3,4\} & 0 & $[7^\omega]+[7^2]$ \\
17 & 16 & \{16\} & 1 & $[7^8]+1$ \\
19 & 3 & \{3\} & 1 & $[7^\omega]+1$ \\
23 & 22 & \{11\} & 1 & $[7^{\omega^4}]+1$ \\
29 & 7 & \{7\} & 0 & $[7^{\omega^3}]$ \\
31 & 15 & \{5,3\} & 0 & $[7^{\omega^2}]+[7^\omega]$ \\
37 & 9 & \{9\} & 1 & $[7^{\omega\cdot\ 3}]+1 $ \\
41 & 40 & \{5,8\} & 0 & $[7^{\omega^2}]+[7^4]$ \\
43 & 6 & \{3,2\} & 3 & $[7^\omega]+7+3$ \\
\hline
\end{tabular}
\end{table}

\begin{table}
\caption[]{$\alpha_u\in \textbf{On}_{11}$} \label{On11}
\begin{tabular}{|c|c|c|c|c|}
\hline $u$ & $f(u)$ & $Q(f(u))$ & excess & $\alpha_u$ \\
\hline 2 & 1 & $\emptyset$ & 2 & 2 \\
3 & 2 & \{2\} & 1 & $11+1$ \\
5 & 1 & $\emptyset$ & 2 & 2 \\
7 & 3 & \{3\} & 1 & $[11^\omega]+1$ \\
13 & 12 & \{3,4\} & 0 & $[11^\omega]+[11^2]$ \\
17 & 16 & \{16\} & 1 & $[11^8]+1$ \\
19 & 3 & \{3\} & 1 & $[11^\omega]+1$ \\
23 & 22 & \{11,2\} & 0 & $[11^{\omega^4}]+11$ \\
29 & 28 & \{7,4\} & 0 & $[11^{\omega^3}]+[11^2]$ \\
31 & 30 & \{5,3\} & 0 & $[11^{\omega^2}]+[11^\omega]$ \\
37 & 6 & \{3\} & 1 & $[11^\omega]+1$ \\
41 & 40 & \{5,8\} & 0 & $[11^{\omega^2}]+[11^4]$ \\
43 & 7 & \{7\} & 1 & $[11^{\omega^3}]+1$ \\
\hline
\end{tabular}
\end{table}

\section{Conclusion: thoughts about $\textbf{On}_0$}

To conclude, let's consider how we might turn the ordinals into $\textbf{On}_0$, a Field of characteristic zero. The inductive construction from Section 1 works just as well in the characteristic zero case, so all that needs to be determined is how to define the smallest field $\phi_0$ (which will be isomorphic to $\QQ$, the field of rationals).

We'll fill in the addition table first, then the multiplication table; for each possible sum or product, we'll choose the smallest ordinal that still allows us to construct a Field of characteristic zero. First of all, we let $0+0=0$; that forces $0$ to be the additive identity, so we have $0+\alpha=\alpha$ for all ordinals $\alpha$. Next, we cannot have $1+1=0$ (or else we no longer have characteristic zero) or $1+1=1$ (since $1$ is not the additive identity), but we can (and must) have $1+1=2$. We then have $1+2=3$, $1+3=4$, and so on; $1+\alpha=[\alpha+1]$ for all $\alpha\in\omega$. But that forces the rest of the addition table for all finite ordinals: for $\alpha,\beta\in\omega$, we have $\alpha+\beta=[\alpha+\beta]$. We just have ordinary addition (and hence, ordinary multiplication) within $\omega$.

Now, the next sum to determine is the sum of $\omega$ and $1$. We can have $\omega+1=0$; $\omega$ can play the role of $-1$. Thus, $\omega+2=1$, $\omega+3=2$, etc. The next undetermined sum is then $\omega+\omega$; this cannot equal $\omega$ or anything in $\omega$, so we must have $\omega+\omega=[\omega+1]$. So $[\omega+1]=-2$, and similarly $[\omega+2]=-3$, and so on. We thus obtain our first group: $[\omega\cdot 2]$ is isomorphic to $\ZZ$, the ring of integers.

Since $[\omega\cdot 2]$ is a group, $[\omega\cdot 2]+\alpha$ cannot be an element of $[\omega\cdot 2]$ for any $\alpha\in [\omega\cdot 2]$. So making the simplest choices at every step, we let $[\omega\cdot 2]+1=[\omega\cdot 2+1]$, $[\omega\cdot 2]+2=[\omega\cdot 2+2]$, and so on; we then have $[\omega\cdot 2]+\alpha=[\omega\cdot 2+\alpha]$ for all $\alpha\in[\omega\cdot 2]$. The next sum to consider is
$[\omega\cdot 2]+[\omega\cdot 2]$. This sum cannot be 0, but it can (and thus must) be 1, since $[\omega\cdot 2]$ may play the role of $1/2$. We thus get $[\omega\cdot 4]$ as our next group, consisting of all halves of integers. We similarly get $[\omega\cdot 4]=1/4$, $[\omega\cdot 8]=1/8$, and so on; $[\omega^2]$ is our next ring, the ring of dyadic rationals.

We then have $[\omega^2]+\alpha=[\omega^2+\alpha]$ for all $\alpha\in[\omega^2]$, so the next sum to consider is $[\omega^2]+[\omega^2]$. This cannot be anything in $[\omega^2]$, since all elements of $[\omega^2]$ already have halves; thus, we must have $[\omega^2]+[\omega^2]=[\omega^2\cdot 2]$. However, we can (and must) have $[\omega^2]+[\omega^2\cdot 2]=1$, letting $[\omega^2]$ play the role of $1/3$. Similarly, we have $[\omega^2\cdot 3]=1/9$, $[\omega^2\cdot 9]=1/27$, etc. Our next ring is then $[\omega^3]=1/5$, and we then have $[\omega^4]=1/7$, $[\omega^5]=1/11$, and so on; for any odd prime $p$, if $p$ is the $k$th prime, then $[\omega^k]=1/p$. Our smallest field is thus $[\omega^\omega]$, which is isomorphic to $\QQ$.

With the smallest field $\phi_0$ thus constructed, we then use the same construction as for $\textbf{On}_p$ to construct all of $\textbf{On}_0$. Unfortunately, further analysis of $\textbf{On}_0$ would seem to be very difficult. We were able to obtain a nearly complete analysis of $\textbf{On}_p$ below the first transcendental, mostly because all elements below $[\omega^{\omega^\omega}]$ are contained in finite fields. But there are no finite fields of characteristic zero. I would imagine that finding the first transcendental in $\textbf{On}_0$ would be as difficult as finding the second transcendental in any $\textbf{On}_p$. So we won't analyze $\textbf{On}_0$ any further.

\end{document}